\documentstyle[amssymb,amsfonts,12pt]{amsart}

\newtheorem{theorem}{Theorem}[section]
\newtheorem{lemma}[theorem]{Lemma}
\newtheorem{corollary}[theorem]{Corollary}
\newtheorem{proposition}[theorem]{Proposition}
\newtheorem{definition}[theorem]{Definition}

\theoremstyle{remark}

\newtheorem*{ack*}{Acknowledgment}

\def\v{{\bf v}}

\def\F{{\mathcal F}}
\def\R{{\mathbb R}}

\def\C{{\mathbb C}}

\def\P{{\mathcal P}}

\def\L{{\mathcal A}}
\def\A{{\mathcal N}}
\def\dist{{\operatorname{dist}}}

\def\bas{\begin{align*}}
\def\eas{\end{align*}}
\def\bi{\begin{itemize}}
\def\ei{\end{itemize}}
\newenvironment{proof}{\noindent {\bf Proof} }{\endprf\par}
\def \endprf{\hfill  {\vrule height6pt width6pt depth0pt}\medskip}
\def\emph#1{{\it #1}}

\begin{document}

\title[Decouplings for surfaces in $\R^4$]{Decouplings for surfaces in $\R^4$}
\author{Jean Bourgain}
\address{School of Mathematics, Institute for Advanced Study, Princeton, NJ 08540}
\email{bourgain@@math.ias.edu}
\author{Ciprian Demeter}
\address{Department of Mathematics, Indiana University, 831 East 3rd St., Bloomington IN 47405}
\email{demeterc@@indiana.edu}

\thanks{The first author is partially supported by the NSF grant DMS-1301619. The second  author is partially supported  by the NSF Grant DMS-1161752}
\begin{abstract}
We prove a sharp decoupling for nondegenerate surfaces in $\R^4$. This puts the progress in \cite{Bo} on the Lindel\"of hypothesis into a more general perspective.
\end{abstract}
\maketitle

\section{Introduction}
Let $\Phi=(\phi_1,\ldots,\phi_4):[0,1]\to\R^4$ be a nondegenerate $C^4$ curve, in the sense  that
\begin{equation}
\label{fe5}
\inf _{t_i\in[0,1]}|\operatorname{det}[\phi_j^{(i)}(t_i)]_{1\le i,j\le 4}|>0.
\end{equation}We will always write $$e(z)=e^{2\pi i z},$$
and for a positive weight $v$
$$\|f\|_{L^p(v)}=(\int|f(x)|^pv(x)dx)^{1/p}.$$
Also, for each ball $B$ centered at $c$ and with radius $R$, $w_{B}$ will denote the weight
$$w_{B}(x)= \frac{1}{(1+\frac{|x-c|}{R})^{100}}.$$
The following result was proved in \cite{Bo}.
\begin{theorem}
\label{thm1}
Let $I_1,I_2$ be intervals in $[0,1]$ with $\dist(I_1,I_2)\sim 1$ and consider  a partition of $[0,1]$ into intervals $I_\tau$ of length $N^{-1/2}$. Then for each $h_i:I_i\to\C$ and each ball $B_N\subset \R^4$ with radius $N$
$$\|(\prod_{i=1}^2|\int_{I_i}h_i(t)e(x\cdot\Phi(t))dt|)^{1/2}\|_{L^{12}(B_{N})}\lesssim_\epsilon N^{-1/6+\epsilon}(\prod_{i=1}^2\sum_{I_\tau}\|\int_{I_\tau}h_i(t)e(x\cdot\Phi(t))dt\|_{L^{6}(w_{B_{N}})}^6)^{\frac{1}{12}}.$$
\end{theorem}

The inequality in Theorem \ref{thm1} played a key role in \cite{Bo} in the estimates on the Riemann zeta function on the critical line. In this paper we offer a different perspective on Theorem \ref{thm1} and show that it is a consequence of the decoupling theory for surfaces in $\R^4$ that we will develop here.

More precisely, consider a compact $C^3$ surface in $\R^4$
$$\Psi(t,s)=(\psi_1(t,s),\ldots,\psi_4(t,s)),\;\;\;\;\psi_i:[0,1]^2\to \R$$
which is assumed to satisfy the nondegeneracy condition
\begin{equation}
\label{fe1}
\operatorname{rank}[\Psi_t(t,s),\Psi_s(t,s),\Psi_{tt}(t,s),\Psi_{ss}(t,s),\Psi_{ts}(t,s)]=4,
\end{equation}
for each $t,s$. Our main new result is the following theorem.
\begin{theorem}
\label{ft1}For each $p\ge 2$,  $g:[0,1]^2\to\C$ and each ball $B_N\subset\R^4$ with radius $N$
$$
\|\int_{[0,1]^2}g(t,s)e(x\cdot\Psi(t,s))dtds\|_{L^p(w_{B_{N}})}\le $$
\begin{equation}
\label{fe3}
\le D(N,p)(\sum_{\Delta\subset [0,1]^2\atop{l(\Delta)=N^{-1/2}}}\|\int_{\Delta}g(t,s)e(x\cdot\Psi(t,s))dtds\|_{L^p(w_{B_{N}})}^p)^{1/p}
\end{equation}
 where the sum on the right hand side is over a partition of $[0,1]^2$ into squares $\Delta$ with side length $l(\Delta)=N^{-1/2}$, and
$$D(N,p)\lesssim_\epsilon N^{\frac12-\frac1p+\epsilon},\,2\le p\le 6$$
$$D(N,p)\lesssim_\epsilon N^{1-\frac4p+\epsilon},\,p\ge 6$$
\end{theorem}
\bigskip

A standard computation with $g=1_{[0,1]^2}$ reveals that the Theorem is essentially sharp, more precisely
\begin{equation}
\label{fe30}
D(N,p)\gtrsim N^{\frac12-\frac1p}\text{ for }2\le p\le 6,\;\;D(N,p)\gtrsim N^{1-\frac4p}\text{ for }6\le p\le \infty.
\end{equation}
It is worth recording the following trivial upper bound that follows from the Cauchy--Schwartz inequality
\begin{equation}
\label{fe3008}
D(N,p)\lesssim N^{1-\frac1p},\text{ for } p\ge 1.
\end{equation}
We will prove that $D(N,6)\lesssim_\epsilon N^{\frac16+\epsilon}$. The estimates for other $p$  will follow by interpolation with the trivial $p=2$ and $p=\infty$ results.
\bigskip

An equivalent way of describing a  surface that satisfies \eqref{fe1} is the fact that it is locally non flat, in the following sense. For each $(t_0,s_0)\in [0,1]^2$ there is no unit vector $\gamma\in\R^4$ such that
$$|\langle\gamma,\Psi(t_0+\Delta t,s_0+\Delta s)-\Psi(t_0,s_0)\rangle|=O(|(\Delta t,\Delta s)|^3), \;\;\;\text{when } |(\Delta t,\Delta s)|\;\to 0.$$
This is easily seen by using Taylor's formula with third order error terms.
\bigskip

Near every  $(t_0,s_0)\in [0,1]^2$, the surface can be represented  with respect to an appropriate system of coordinates as
\begin{equation}
\label{fe11}
(t,s, A_1t^2+2A_2ts+A_3s^2, A_4t^2+2A_5ts+A_6s^2)+O(|(t,s)|^3).
\end{equation}
It suffices to choose two perpendicular axes in the tangent plane at $(t_0,s_0)$  and to apply the Taylor expansion for the surface with respect to these variables. It is easy to see that $\Psi$ satisfies \eqref{fe1} at $(t_0,s_0)$ if and only if
\begin{equation}
\label{fe12}
\operatorname{rank}\left[ \begin{array}{ccc}
A_1 & A_2 & A_3 \\
A_4 & A_5 & A_6 \end{array} \right]=2
\end{equation}
\bigskip

Let us now place \eqref{fe3} in the context of the more general decoupling theory from  \cite{BD3} and \cite{BD4}. These papers completely settle the case of hypersurfaces in all dimensions, whose Gaussian curvature is nonzero everywhere. In two dimensions, the sharp inequality takes the form
$$\|\int_{[0,1]}h(t)e(x_1t+x_2\gamma(t))dt\|_{L^6(w_{B_N})}\lesssim_\epsilon $$
\begin{equation}
\label{fe6}
N^\epsilon(\sum_\tau\|\int_{I_\tau}h(t)e(x_1t+x_2\gamma(t))dt\|_{L^6(w_{B_N})}^2)^{1/2},
\end{equation}
where $I_\tau$ are intervals of length  $N^{-1/2}$ that partition $[0,1]$ and $\gamma:[0,1]\to\R$ satisfies the curvature condition  $$\inf_{0\le t\le 1}|\gamma''(t)|>0.$$

The paper \cite{BD4} also  analyzes the decoupling theory for curves satisfying \eqref{fe5} in arbitrary dimensions $n$, but the picture for $n\ge 3$ is rather incomplete. The natural scaling for curves dictates that the intervals $I_\tau$ should have larger length $N^{-1/n}$, in order to run the machinery from \cite{BD3}. However,   Theorem \ref{thm1} goes against this principle and uses intervals $I_\tau$ of length $N^{-1/2}$, rather than $N^{-1/4}$. One of the goals of this paper is to clarify why it is possible to work with such a scale.

While \cite{BD3} and \cite{BD4} deal with manifolds of codimension one and dimension one, respectively,  inequality \eqref{fe3} is the first example that does not fall in either category.  We will argue that the natural scale to work with  in the context of surfaces in $\R^4$ is $N^{-1/2}$, the same as for hypersurfaces in any dimension. This will allow for the argument  in \cite{BD3} to be carried over in this context. Moreover, Theorem \ref{thm1} will be seen to be a rather immediate consequence of  Theorem \ref{ft1}. See Section \ref{fs2} below.

\bigskip

Inequality \eqref{fe3} at the critical index $p=6$ is an $l^6(L^6)$ decoupling, a bit weaker than the $l^2(L^6)$ decoupling
$$
\|\int_{[0,1]^2}g(t,s)e(x\cdot\Psi(t,s))dtds\|_{L^6(w_{B_{N}})}\lesssim_\epsilon
$$\begin{equation}
\label{fe2}
N^{\epsilon}(\sum_{\Delta\subset [0,1]^2\atop{l(\Delta)=N^{-1/2}}}\|\int_{[0,1]^2}g(t,s)e(x\cdot\Psi(t,s))dtds\|_{L^6(w_{B_{N}})}^2)^{1/2}
\end{equation}
since  \eqref{fe2}  automatically implies \eqref{fe3} via H\"older's inequality. We will see below however  that   \eqref{fe2} is false for some $\Psi$ satisfying \eqref{fe1}. Indeed, a standard discretization argument as in \cite{BD3} shows that if  \eqref{fe2} were true, it would imply the following estimate for exponential sums
\begin{equation}
\label{fe4}
\frac{1}{|B_N|^{1/6}}\|\sum_{n=1}^Na_ne(x\cdot\Psi(\xi_n))\|_{L^6(B_N)}\lesssim_\epsilon N^{\epsilon}\|a_n\|_{l^2},
\end{equation}
for arbitrary $a_n\in \C$ and $N^{-1/2}$ separated $\xi_n\in[0,1]^2$. Choose now
$$\Psi(t,s)=(t,s,t^2,ts),$$
which is easily seen to satisfy \eqref{fe1}. This surface contains the line  $$\{(0,s,0,0):s\in\R\}.$$ Choose $N^{1/2}$ equidistributed points $\xi_n$ on the line segment $\{0\}\times [0,1]$ which are $N^{-1/2}$ separated, and let $a_n=1$. It is easy to see that \eqref{fe4} will fail for this choice.  We point out that \eqref{fe2} holds true for other $\Psi$ such as
$$\Psi(t,s)=(t,s,t^2,s^2),$$
a consequence of \eqref{fe6} (applied twice) and Fubini. This easy argument is not applicable for more ``twisted" $\Psi$, but at least explains the $L^6$ numerology, and shows  that  \eqref{fe3} generalizes \eqref{fe6}. It will later become clear that the proof of \eqref{fe6}  in \cite{BD3} is a good template for our  proof of \eqref{fe3}. In particular, both inequalities rely crucially on a bilinear estimate.

Theorem \ref{ft1} has natural extensions to the case when $\Psi:[0,1]^k\to\R^{2k}$, $k\ge 3$, but we will not pursue them here. The decoupling constants will depend on the dimension of the largest affine subspace of the manifold $\Psi$. This is similar to the role played by signature in the decoupling theory for hyperbolic paraboloids developed in \cite{BD4}.
\bigskip

\section{The proof of Theorem \ref{thm1}}
\label{fs2}

This section will be devoted to proving  Theorem \ref{thm1} assuming Theorem \ref{ft1}, for $p=6$.

Given a curve $\Phi$ satisfying \eqref{fe5} and intervals $I_1,I_2$ as in Theorem \ref{thm1}, we construct the surface

\begin{equation}
\label{fe8}
\Psi_\Phi(t,s)=(\phi_1(t)+\phi_1(s),\ldots, \phi_4(t)+\phi_4(s)),\;\;t\in I_1,\;s\in I_2.
\end{equation}
We prove  that $\Psi_\Phi$ satisfies \eqref{fe1} on $I_1\times I_2$. It will suffice if we show that
\begin{equation}
\label{fe7}
\operatorname{det}\left[ \begin{array}{cccc}
\phi_1'(t) & \phi_2'(t) & \phi_3'(t) & \phi_4'(t) \\ \phi_1'(s) & \phi_2'(s) & \phi_3'(s) & \phi_4'(s)\\ \phi_1''(t) & \phi_2''(t) & \phi_3''(t) & \phi_4''(t) \\ \phi_1''(s) & \phi_2''(s) & \phi_3''(s) & \phi_4''(s)
\end{array} \right]\not =0,\,\,\text{for }t\in I_1,\;s\in I_2.
\end{equation}
Note that $$\operatorname{det}\left[ \begin{array}{cccc}
\phi_1'(t) & \phi_2'(t) & \phi_3'(t) & \phi_4'(t) \\ \phi_1'(s) & \phi_2'(s) & \phi_3'(s) & \phi_4'(s)\\ \phi_1''(t) & \phi_2''(t) & \phi_3''(t) & \phi_4''(t) \\ \phi_1''(s) & \phi_2''(s) & \phi_3''(s) & \phi_4''(s)
\end{array} \right]=\lim_{\epsilon\to 0}\frac1{\epsilon^2}\operatorname{det}\left[ \begin{array}{cccc}
\phi_1'(t) & \phi_2'(t) & \phi_3'(t) & \phi_4'(t) \\ \phi_1'(s) & \phi_2'(s) & \phi_3'(s) & \phi_4'(s)\\ \phi_1'(t+\epsilon) & \phi_2'(t+\epsilon) & \phi_3'(t+\epsilon) & \phi_4'(t+\epsilon) \\ \phi_1'(s+\epsilon) & \phi_2'(s+\epsilon) & \phi_3'(s+\epsilon) & \phi_4'(s+\epsilon)
\end{array} \right].$$
A generalization of the Mean-Value Theorem (see \cite{PS}, Voll II, part V, Chap 1, No. 95) guarantees that
$$\operatorname{det}\left[ \begin{array}{cccc}
\phi_1'(t) & \phi_2'(t) & \phi_3'(t) & \phi_4'(t) \\ \phi_1'(s) & \phi_2'(s) & \phi_3'(s) & \phi_4'(s)\\ \phi_1'(t+\epsilon) & \phi_2'(t+\epsilon) & \phi_3'(t+\epsilon) & \phi_4'(t+\epsilon) \\ \phi_1'(s+\epsilon) & \phi_2'(s+\epsilon) & \phi_3'(s+\epsilon) & \phi_4'(s+\epsilon)
\end{array} \right]=$$$$\epsilon^2(t-s)^2(t+\epsilon-s)(s+\epsilon-t)\operatorname{det}\left[ \begin{array}{cccc}
\phi_1'(\tau_1) & \phi_2'(\tau_1) & \phi_3'(\tau_1) & \phi_4'(\tau_1) \\ \phi_1''(\tau_2) & \phi_2''(\tau_2) & \phi_3''(\tau_2) & \phi_4''(\tau_2)\\ \phi_1'''(\tau_3) & \phi_2'''(\tau_3) & \phi_3'''(\tau_3) & \phi_4'''(\tau_3) \\ \phi_1''''(\tau_4) & \phi_2''''(\tau_4) & \phi_3''''(\tau_4) & \phi_4''''(\tau_4)
\end{array} \right]$$
for some $\tau_i\in[0,1]$ depending on $t,s,\epsilon$. It is now clear that \eqref{fe7} follows from \eqref{fe5}.
\bigskip

For an arbitrary $S\subset [0,1]^2$, $g:S\to\C$ and $\Psi$ satisfying  \eqref{fe1} define the extension operator
$$E_{S,\Psi}g(x)=\int_S g(t,s)e(x\cdot \Psi(t,s))dtds.$$
Going back to Theorem \ref{thm1} we note that
$$\prod_{i=1}^2\int_{I_i}h_i(t)e(x\cdot\Phi(t))dt=E_{I_1\times I_2,\Psi_\Phi}g(x)$$
with $g(t,s)=h_1(t)h_2(s).$ By applying Theorem \ref{ft1} it follows that
$$\|(\prod_{i=1}^2|\int_{I_i}h_i(t)e(x\cdot\Phi(t))dt|)^{1/2}\|_{L^{12}(B_{N})}\lesssim_\epsilon N^{\frac16+\epsilon}(\sum_{\Delta\subset [0,1]^2\atop{l(\Delta)=N^{-1/2}}}\|E_{\Delta,\Psi_\Phi}g\|_{L^6(w_{B_{N}})}^6)^{1/12}.$$
\bigskip

Fix $\Delta=J_1\times J_2$ from the summation above. To prove Theorem \ref{thm1}, it will suffice to argue that
$$\|E_{\Delta,\Psi_\Phi}g\|_{L^6(w_{B_{N}})}\lesssim$$
\begin{equation}
\label{fe9}
\lesssim N^{-\frac23}\|\int_{J_1}h_1(t)e(x\cdot\Phi(t))dt\|_{L^6(w_{B_{N}})}\|\int_{J_2}h_2(t)e(x\cdot\Phi(t))dt\|_{L^6(w_{B_{N}})}.
\end{equation}
This inequality will follow from the following transversality result.

\begin{lemma}Let $T_1,T_2$ be two cylindrical tubes in $\R^4$ with  length $\sim N^{-1/2}$ in the direction  $\v_i$ and radius $\sim N^{-1}$. Assume the angle between $\v_1$ and $\v_2$ is $\sim 1$. Let $f_i$ be a function which is Fourier supported in $T_i$. Then
$$\|f_1f_2\|_{L^2(\R^4)}\lesssim N^{-2}\|f_1\|_{L^2(\R^4)}\|f_2\|_{L^2(\R^4)}.$$
\end{lemma}
\begin{proof}
Start with a wave packet decomposition
$$f_i=\sum_{P\in\P_i}a_P\phi_P,$$
where $\P_i$ is a finitely overlapping cover of $\R^4$ with $N\times N\times N\times N^{1/2}$ rectangular plates $P_i$ orthogonal to $\v_i$. In particular, we may assume that
$$\int|\phi_P|^2=1,\;\;P\in\P_1\cup\P_2$$
$$\|f_i\|_2\sim(\sum_{P\in\P_i}|a_P|^2)^{1/2},$$
and also that $\phi_P$ has rapid decay away from $P$. The transversality guaranteed by the angle between $\v_1,\v_2$ forces
$$\int|\phi_{P_1}\phi_{P_1'}\phi_{P_2}\phi_{P_2'}|\le N^{-4}c(P_1,P_1')c(P_2,P_2'), \;\;P_i,P_i'\in\P_i,$$
where $c(P_i,P_i')$ are weights that decrease rapidly with the distance between $P_i,P_i'$. We conclude that
$$\|f_1f_2\|_{L^2(\R^4)}^2\le N^{-4}\sum_{P_1,P_1'\in\P_1}|a_{P_1}a_{P_1'}c(P_1,P_1')|\sum_{P_2,P_2'\in\P_2}|a_{P_2}a_{P_2'}c(P_2,P_2')|.$$
By a few applications of H\"older's inequality, this is further bounded by
$$N^{-4}\sum_{P_1\in\P_1}|a_{P_1}|^2\sum_{P_2\in\P_2}|a_{P_2}|^2,$$
as desired.
\end{proof}
\bigskip

To see \eqref{fe9}, consider a positive weight $v_{B_N}$ in the Schwartz class with Fourier support in the $N^{-1}$ neighborhood of the origin in $\R^4$ and which is $\ge 1$ on $B_N$. Note that the functions
$$f_i(x)=(v_{B_N}(x)\int_{J_i}h_i(t)e(x\cdot\Phi(t))dt)^3$$
satisfy the requirements of the lemma from above. We get
$$\|E_{\Delta,\Psi_\Phi}g\|_{L^6(B_{N})}\lesssim \|f_1f_2\|_{L^2(\R^4)}^{1/3}\lesssim$$$$\lesssim N^{-2/3}\|\int_{J_1}h_1(t)e(x\cdot\Phi(t))dt\|_{L^6(v^6_{B_{N}})}\|\int_{J_2}h_2(t)e(x\cdot\Phi(t))dt\|_{L^6(v^6_{B_{N}})}.$$
Inequality \eqref{fe9} now follows from standard manipulations, by using the fact that $v_{B_N}$ has Schwartz decay, while $w_{B_N}$ has prescribed polynomial decay.

\section{Reduction to quadratic surfaces}
This section will clarify why in Theorem \ref{ft1} the decoupling intervals $\Delta$ have scale $N^{-1/2}$. The key will be the approximation of nondegenerate surfaces by quadratic ones.

For a $\Psi$ satisfying \eqref{fe1} we let $\A_N=\A_N(\Psi)$ be the $N^{-1}$ neighborhood of the surface $\Psi([0,1]^2)$. Consider a fixed finitely overlapping cover of $[0,1]^2$ with squares $\Delta$ of side length $N^{-1/2}$, and let $\P_N$ be the associated cover  of $\A_N$ with $N^{-1}$ neighborhoods $\theta$ of $\Psi(\Delta)$. Note that each $\theta$ is essentially a rectangular region with dimensions $N^{-1},N^{-1}, N^{-1/2}, N^{-1/2}$. We will denote by $f_\theta$ an appropriate smooth Fourier restriction of $f$ to $\theta$ so that
$$f=\sum_{\theta\in \P_N}f_\theta.$$

Let $K_{p,\Psi}(N)$ denote the best constant such that
\begin{equation}
\label{fe18}
\|f\|_{L^p(\R^4)}\le K_{p,\Psi}(N)(\sum_{\theta\in \P_N}\|f_\theta\|_{L^p(\R^4)}^p)^{1/p}
\end{equation}
holds for each $f$ Fourier supported in $\A_N(\Psi)$.
We first observe that  Theorem \ref{ft1} is equivalent with proving that
\begin{equation}
\label{fe10}
K_{6,\Psi}(N)\lesssim_\epsilon N^{\frac13+\epsilon}.
\end{equation}
This can be seen by foliating $\A_N(\Psi)$ into translates of $\Psi$.  We refer the reader to \cite{BD3} for  details for a similar statement in a related context.

Next, we will prove that \eqref{fe10} for a given $\Psi$  follows if we assume \eqref{fe10} for quadratic surfaces of the form
$$\Psi_{\bf{A}}(t,s)=(t,s, A_1t^2+2A_2ts+A_3s^2, A_4t^2+2A_5ts+A_6s^2),$$
with ${\bf A}=(A_1,\ldots,A_6)$ satisfying \eqref{fe12}. To see this, let $f$ be Fourier supported in $\A_N(\Psi)$. Since $\widehat{f}$ is also Fourier supported in  $\A_{N^{\frac23}}(\Psi)$, we have the initial decoupling
\begin{equation}
\label{fe14}
\|f\|_{L^p(\R^4)}\le K_{p,\Psi}(N^{\frac23})(\sum_{\tau\in \P_{N^{\frac23}}}\|f_\tau\|_{L^p(\R^4)}^p)^{1/p}.
\end{equation}
Note that each $f_\tau$ is  Fourier supported on $\tau\cap \A_N(\Psi)$. Also,  \eqref{fe11} shows that after a rotation  $\tau\cap \A_N(\Psi)$ is a subset of $\A_{O(N)}(\Psi_{\bf A})$ for some ${\bf A}$ satisfying \eqref{fe12}. By applying the linear rescaling
$$(t_1,\ldots,t_4)\mapsto(N^{1/3}t_1,N^{1/3}t_2,N^{2/3}t_3,N^{2/3}t_4)$$
with respect to the local system of coordinates,
 $\tau\cap \A_N(\Psi)$ is mapped into $\A_{O(N^{1/3})}(\Psi_{\bf A})$.
By rescaling back, it follows that
\begin{equation}
\label{fe15}
\|f_\tau\|_{L^p(\R^4)}\le K_{p,\Psi_{\bf A}}(N^{1/3})(\sum_{\theta\in \P_N\atop{\theta\subset \tau}}\|f_\theta\|_{L^p(\R^4)}^p)^{1/p}.
\end{equation}
The vector ${\bf A}$ will of course depend on $\tau$. Combining \eqref{fe14} and \eqref{fe15} we conclude that there exists $C_\Psi$ depending only on $\Psi$, such that for each $N$
\begin{equation}
\label{fe13}
K_{p,\Psi}(N)\le K_{p,\Psi}(N^{\frac23})\sup_{{\bf A}\in\L}K_{p,\Psi_{\bf A}}(N^{\frac13})
\end{equation}
with $$\L=$$$$\left\{{\bf A}:\;|A_i|\le C_\Psi, \;\max\left\{\left|\operatorname{det}\left[ \begin{array}{ccc}
A_1 & A_2  \\
A_4 & A_5 \end{array} \right]\right|, \left|\operatorname{det}\left[ \begin{array}{ccc}
A_1 & A_3  \\
A_4 & A_6 \end{array} \right]\right|, \left|\operatorname{det}\left[ \begin{array}{ccc}
A_3 & A_2  \\
A_6 & A_5 \end{array} \right]\right| \right\}\ge C_\Psi^{-1}\right\}.$$
Thus, assuming
$$\sup_{{\bf A}\in\L}K_{6,\Psi_{\bf A}}(N)\lesssim_{\epsilon,C_\Psi} N^{\frac13+\epsilon},$$
\eqref{fe10} will follow by iterating \eqref{fe13}.

From now on we will work with the surface $\Psi_{\bf A}$, for a fixed ${\bf A}\in \L$, and will show that
$$K_{6,\Psi_{\bf A}}(N)\lesssim_\epsilon N^{\frac13+\epsilon},$$
A careful analysis of the forthcoming argument will show that  the implicit constant will depend on $C_\Psi$. We will never specify the exact dependence.

\bigskip

\section{Transversality and a bilinear theorem}
Fix ${\bf A}\in \L$. The extension operator $E$ will from now on be implicitly understood to be with respect to $\Psi_{\bf A}$. Define
$$c_{1,{\bf A}}=\operatorname{\det} \left[ \begin{array}{ccc}
A_1 & A_2  \\
A_4 & A_5 \end{array} \right], \;\;c_{2,{\bf A}}=\operatorname{det}\left[ \begin{array}{ccc}
A_1 & A_3  \\
A_4 & A_6 \end{array} \right], \;\;c_{3,{\bf A}}=\operatorname{det}\left[ \begin{array}{ccc}
A_6 & A_5  \\
A_3 & A_2 \end{array} \right].$$
\begin{definition}
Let $\nu\le 1$. We say that two sets $S_1,S_2\subset [0,1]^2$ are  $\nu$-transverse if
\begin{equation}
\label{fe16}
c_{1,{\bf A}}(t_1-t_2)^2+c_{2,{\bf A}}(t_1-t_2)(s_1-s_2)+ c_{3,{\bf A}}(s_1-s_2)^2\ge \nu
\end{equation}
for each $(t_i,s_i)\in S_i.$
\end{definition}
The following bilinear theorem will play a key role in our approach.
\begin{theorem}
\label{ft3}
Let $R_1$, $R_2$ be two $\nu$-transverse squares in $[0,1]^2$. Then for each $g_i:R_i\to\C$ we have
$$\||E_{R_1}g_1E_{R_2}g_2|^{1/2}\|_{L^4(\R^4)}\lesssim_\nu (\|g_1\|_{L^2(R_1)}\|g_2\|_{L^2(R_2)})^{1/2}.$$
\end{theorem}
\begin{proof}
We will perform the following change of variables
$$(t_1,s_1,t_2,s_2)\in R_1\times R_2\mapsto (u_1,u_2,u_3,u_4)=$$$$(t_1+t_2,s_1+s_2,A_1(t_1^2+t_2^2)+2A_2(t_1s_1+t_2s_2)+A_3(s_1^2+s_2^2), A_4(t_1^2+t_2^2)+2A_5(t_1s_1+t_2s_2)+A_6(s_1^2+s_2^2))$$
whose Jacobian is
$$c_{1,{\bf A}}(t_1-t_2)^2+c_{2,{\bf A}}(t_1-t_2)(s_1-s_2)+ c_{3,{\bf A}}(s_1-s_2)^2\ge \nu.$$
It follows that
$$|E_{R_1}g_1(x)E_{R_2}g_2(x)|=\widehat{GJ}(x)$$
where $G(u_1,\ldots,u_4)=g_1(t_1,s_1)g_2(t_2,s_2)$ and $|J(u)|\le \nu^{-1}$.

Using Plancherel's identity we get
$$\||E_{R_1}g_1E_{R_2}g_2|^{1/2}\|_{L^4(\R^4)}=\|GJ\|_{L^2(\R^4)}^{1/2}\le$$$$\le \nu^{-\frac14}(\int_{\R^4}|G^2(u)J(u)|du)^{\frac14}=\nu^{-\frac14}(\|g_1\|_{L^2(R_1)}\|g_2\|_{L^2(R_2)})^{1/2}.$$
\end{proof}
\begin{corollary}
\label{fc1}
Let $R_1,R_2\subset [0,1]^2$ be $\nu$-transverse squares. Then for each $4\le p\le \infty$ and  $g_i:R_i\to\C$ we have
\begin{equation}
\label{we3}
\|(\prod_{i=1}^2\sum_{\atop{l(\Delta)=N^{-1/2}}} |E_{\Delta}g_i|^2)^{1/4}\|_{L^{p}(w_{B_N})}\lesssim_\nu N^{-\frac4{p}}(\prod_{i=1}^2\sum_{\atop{l(\Delta)=N^{-1/2}}}\|E_{\Delta}g_i\|_{L^{p/2}(w_{B_N})}^2)^{\frac1{4}}.
\end{equation}
\end{corollary}
\begin{proof}
A standard consequence of Theorem \ref{ft3} and Plancherel's identity is the following local inequality
$$\|(\prod_{i=1}^2 |E_{R_i}g_i|)^{1/2}\|_{L^{4}(w_{B_N})}\lesssim_\nu N^{-1}(\prod_{i=1}^2\sum_{\atop{l(\Delta)=N^{-1/2}}}\|E_{\Delta}g_i\|_{L^{2}(w_{B_N})}^2)^{\frac1{4}}.$$
A  randomization argument further leads to the inequality
$$\|(\prod_{i=1}^2\sum_{\atop{l(\Delta)=N^{-1/2}}} |E_{\Delta}g_i|^2)^{1/4}\|_{L^{4}(w_{B_N})}\lesssim_\nu N^{-1}(\prod_{i=1}^2\sum_{\atop{l(\Delta)=N^{-1/2}}}\|E_{\Delta}g_i\|_{L^{2}(w_{B_N})}^2)^{\frac1{4}}.$$
It now suffices to interpolate this with the trivial inequality
$$\|(\prod_{i=1}^2\sum_{\atop{l(\Delta)=N^{-1/2}}} |E_{\Delta}g_i|^2)^{1/4}\|_{L^{\infty}(w_{B_N})}\le (\prod_{i=1}^2\sum_{\atop{l(\Delta)=N^{-1/2}}}\|E_{\Delta}g_i\|_{L^{\infty}(w_{B_N})}^2)^{\frac1{4}}.$$
We refer the reader to \cite{BD3} for how this type of interpolation is performed.
\end{proof}

\bigskip

For the argument in the following sections, it will be important that transversality is quite generic.
\begin{proposition}
\label{fp1}
For $K=2^m\ge 1$, consider the collection $Col_K$ of the $K^2$ dyadic squares in $[0,1]^2$ with side length $K^{-1}$. For each $R\in Col_K$, there are $O(K)$ squares $R'\in Col_K$ which are $K^{-2}$-transverse to $R$.
\end{proposition}
\begin{proof}
Let $\lambda_1,\lambda_2$ be the eigenvalues of $$M_{\bf A}=\left[ \begin{array}{ccc}
c_{1,{\bf A}} & \frac{c_{2,{\bf A}}}2  \\
\frac{c_{2,{\bf A}}}2 & c_{3,{\bf A}} \end{array} \right]$$
and let $\v_1,\v_2$ be a corresponding orthonormal eigenbasis.
Assume $|\lambda_1|\ge |\lambda_2|$. The condition ${\bf A}\in \L$ guarantees that $|\lambda_1|\gtrsim 1$.

If $(\beta_1,\beta_2)$ are the coordinates of $(t_1-t_2,s_1-s_2)$ with respect to the $(\v_1,\v_2)$ basis, then \eqref{fe16} becomes
\begin{equation}
\label{fe17}
|\beta_1^2\lambda_1+\beta_2^2\lambda_2|\ge K^{-2}.
\end{equation}
If $\lambda_1\lambda_2\ge 0$, then $|\beta_1|\ge |\lambda_1|^{-1/2}K^{-1}$ will automatically force \eqref{fe17}. $R$ and $R'$ will be  $K^{-2}$-transverse as soon as $R'$ is outside the $O(K^{-1})$ wide strip containing $R$ and stretching  in the direction of $\v_2$. There are $O(K)$ squares in this strip.

If $\lambda_1\lambda_2<0 $, then $$\min\left\{\left|\beta_1+\beta_2\sqrt{-\frac{\lambda_2}{\lambda_1}}\;\right|,\left|\beta_1-\beta_2\sqrt{-\frac{\lambda_2}{\lambda_1}}\;\right|\right\}\ge |\lambda_1|^{-1/2}K^{-1}$$ will again force \eqref{fe17}. $R$ and $R'$ will be  $K^{-2}$-transverse as soon as $R'$ is outside the two $O(K^{-1})$ wide strips containing $R$ and stretching  in the directions of $\pm\sqrt{-\frac{\lambda_2}{\lambda_1}}\v_1+\v_2$.
\end{proof}

\section{Linear versus bilinear decoupling}

We will make use of the following ``trivial" decoupling to treat the non transverse contribution in the Bourgain--Guth decomposition.
\begin{lemma}
\label{fl1}
Let $R_1,\ldots,R_K$ be pairwise disjoint squares in $[0,1]^2$ with side length $K^{-1}$. Then for each $2\le p\le \infty$
$$ \|\sum_iE_{R_i}g\|_{L^p(w_{B_K})}\lesssim_p K^{1-\frac2p}(\sum_i\|E_{R_i}g\|_{L^p(w_{B_K})}^p)^{1/p}.$$
\end{lemma}
\begin{proof}
The key observation is the fact that if $f_1,\ldots,f_K:\R^4\to\C$ are such that $\widehat{f_i}$ is supported on a ball $B_i$ and the dilated balls $(2B_i)_{i=1}^K$ are pairwise disjoint, then 
\begin{equation}
\label{fe36}
\|f_1+\ldots+f_K\|_{L^p(\R^4)}\lesssim_p  K^{1-\frac2p}(\sum_i\|f_i\|_{L^p(\R^4)}^p)^{\frac1p}.
\end{equation}
In fact more is true. If $T_i$ is a smooth Fourier multiplier adapted to $2B_i$ and equal to 1 on $B_i$, then the inequality
$$\|T_1(f_1)+\ldots+T_K(f_K)\|_{L^p(\R^4)}\lesssim_p  K^{1-\frac2p}(\sum_i\|f_i\|_{L^p(\R^4)}^p)^{\frac1p}$$
for arbitrary $f_i\in L^p(\R^4)$
follows by interpolating the immediate $L^2$ and $L^\infty$ estimates. Inequality \eqref{fe36} is the best one can say in general, if no further assumption is made on the Fourier supports of $f_i$. Indeed, if $\widehat{f_i}=1_{B_i}$ with $B_i$ equidistant balls of radius one with collinear centers, then the reverse inequality will hold.

Let now $v_{B_K}$ be a Schwartz function with Fourier support in the $K^{-1}$ neighborhood of the origin in $\R^4$ and which is $\ge 1$ on $B_K$. It suffices to note that the Fourier supports of the  functions $f_i=v_{B_K}E_{R_i}g$ have bounded overlap.
 \end{proof}
\bigskip

For $2\le p<\infty$ and $N\ge 1$,  recall that $D(N,p)$ is the smallest constant such that the decoupling
$$\|E_{[0,1]^2}g\|_{L^p(w_{B_N})}\le D(N,p)(\sum_{l(\Delta)=N^{-1/2}}\|E_{\Delta}g\|_{L^p(w_{B_N})}^p)^{1/p}$$
holds true for all $g$ and all balls $B_N$ or radius $N$. The sum on the right is over a partition of $[0,1]^2$ into dyadic squares $\Delta$ of side length $N^{-1/2}$.

We now introduce a bilinear version of $D(N,p) $. Given also $\nu\le 1$, let $D_{multi}(N,p,\nu)$ be the smallest constant such that the bilinear decoupling
$$\||E_{R_1}g_1E_{R_2}g_2|^{\frac12}\|_{L^p(w_{B_N})}\le D_{multi}(N,p,\nu)(\prod_{i=1}^2\sum_{l(\Delta)=N^{-1/2}}\|E_{\Delta}g_i\|_{L^p(w_{B_N})}^p)^{\frac1{2p}}$$
holds true for all $\nu$-transverse squares $R_1,R_2\subset [0,1]^{2}$ with arbitrary side lengths, all $g_i:R_i\to\C$  and all balls $B_N\in\R^4$ with radius $N$.

H\"older's inequality shows that $D_{multi}(N,p,\nu)\le D(N,p)$. The rest of the section will be devoted to proving that the reverse inequality is also essentially true. This will  follow from a  variant of the Bourgain--Guth induction on scales in \cite{BG}. More precisely, we prove the following result.

\begin{theorem}
\label{ft2}
For each $\nu\le \frac1{10}$ and $p\ge 2$ there exists $C_{\nu}>0$ and $\epsilon(\nu,p)$  with $\lim_{\nu\to 0}\epsilon(\nu,p)=0$ such that for each $N\ge 1$
\begin{equation}
\label{fe31}
D(N,p)\le C_{\nu}N^{\epsilon(\nu,p)}\sup_{1\le M\le N}(\frac{M}{N})^{\frac1p-\frac12}D_{multi}(M,p,\nu).
\end{equation}
\end{theorem}

The key step in achieving this result is the following inequality.

\begin{proposition}
\label{fp2}
For $2\le p<\infty$, there is a constant $C_p$ which only depends on $p$ so that for each $g$ and $N,K\ge 1$ we have
$$\|E_{[0,1]^2}g\|_{L^p(w_{B_N})}^p\le $$$$ C_p\left(K^{p-2}\sum_{l(R)=\frac1K}\|E_{R}g\|_{L^p(w_{B_N})}^p+K^{4p}D_{multi}(N,p,K^{-2})^p\sum_{l(\Delta)=N^{-1/2}}\|E_{\Delta}g\|_{L^p(w_{B_N})}^p\right).$$
\end{proposition}
\bigskip

The exponent $4p$ in $K^{4p}$ is not important and could easily be improved, but the exponent $p-2$ in $K^{p-2}$ is sharp and will play a critical role in the rest of the argument.
\bigskip

\begin{proof}
Following the standard formalism from \cite{BG}, we may assume that $|E_{R}g(x)|$ is essentially constant on each ball $B_K$ of radius $K$, and will we denote by $|E_{R}g(B_K)|$ this value. Write
$$E_{[0,1]^2}g(B_K)=\sum_{l(R)=\frac1K}E_{R}g(B_K).$$
Fix $B_K$. Let $R^*$ be a square which maximizes the value of $|E_{R}g(B_K)|$. We distinguish two cases.

First, if there is some $R^{**}$ which is $K^{-2}$-transverse to $R^*$ and such that
$|E_{R^{**}}g(B_K)|\ge K^{-2}|E_{R^*}g(B_K)|$, then
$$|E_{[0,1]^2}g(B_K)|\le K^3(|E_{R^{*}}g(B_K)E_{R^{**}}g(B_K)|)^{1/2}.$$

Otherwise, $|E_{R}g(B_K)|<K^{-2}|E_{R^*}g(B_K)|$ whenever $R$ is $K^{-2}$-transverse to $R^*$. In this case we can write
$$|E_{[0,1]^2}g(B_K)|\le 2|E_{R^*}g(B_K)|+|\sum_{R\text{ not } K^{-2}-\text{transverse to }R^* }E_{R}g(B_K)|.$$
Using Lemma \ref{fl1} and Proposition \ref{fp1} we get
$$\|E_{[0,1]^2}g\|_{L^p(w_{B_K})}\lesssim_p \|E_{R^*}g\|_{L^p(w_{B_K})}+ K^{1-\frac2{p}}(\sum_{R\text{ not }K^{-2}-\text{transverse to }R^*}\|E_{R}g\|_{L^p(w_{B_K})}^p)^{1/p}\le $$
$$\lesssim  K^{1-\frac2{p}}(\sum_{\text{all }R}\|E_{R}g\|_{L^p(w_{B_K})}^p)^{1/p}.$$
To summarize, in either case we can write
$$\|E_{[0,1]^2}g\|_{L^p(w_{B_K})}\lesssim_p $$$$ K^3\max_{R_1,R_2:\;K^{-2}-\text{transverse}}\|(|E_{R_1}gE_{R_2}g|)^{1/2}\|_{L^p(w_{B_K})}+ K^{1-\frac2{p}}(\sum_{\text{all }R}\|E_{R}g\|_{L^p(w_{B_K})}^p)^{1/p}\le$$
$$\le K^3(\sum_{R_1,R_2:\;K^{-2}-\text{transverse}}\|(|E_{R_1}gE_{R_2}g|)^{1/2}\|_{L^p(w_{B_K})}^p)^{1/p}+ K^{1-\frac2{p}}(\sum_{\text{all }R}\|E_{R}g\|_{L^p(w_{B_K})}^p)^{1/p}.$$
Raising to the power $p$  and summing over $B_K\subset B_N$ leads to the desired conclusion.
\end{proof}
\bigskip

It is worth mentioning that in general, one can not do better than the trivial decoupling to estimate the contribution from the non transverse terms. This is illustrated by the example
$$\Psi(t,s)=(t,s,t^2,ts).$$
The squares $R_1,\ldots,R_K\subset [0,1]^2$ containing the line segment $\{0\}\times [0,1]$ will be pairwise non transverse. It is easy to see that
$$\|\sum_iE_{R_i}1\|_{L^p(w_{B_K})}\sim_p K^{1-\frac2p}(\sum_i\|E_{R_i}1\|_{L^p(w_{B_K})}^p)^{1/p}.$$
This is due to the fact that the sets $\Psi(R_i)$ intersect the line $$\{(0,s,0,0):s\in\R\}.$$
The lack of curvature prevents anything better than the trivial decoupling to hold.
\bigskip

\bigskip

Using a form of parabolic rescaling,  the result in Proposition \ref{fp2} leads to the following general result.
\begin{proposition}
\label{fp3}
Let $R\subset[0,1]^2$ be a square with side length $\delta$. For $2\le p<\infty$, there is a constant $C_p$ which only depends on $p$ so that for each $g$, $K\ge 1$ and $N>\delta^{-2}$ we have
$$\|E_{R}g\|_{L^p(w_{B_N})}^p\le $$$$ C_p\left(K^{p-2}\sum_{R'\subset R\atop{l(R')=\frac\delta{K}}}\|E_{R'}g\|_{L^p(w_{B_N})}^p+K^{4p}D_{multi}(N\delta^{2},p,K^{-2})^p\sum_{\Delta\subset R\atop{l(\Delta)=N^{-1/2}}}\|E_{\Delta}g\|_{L^p(w_{B_N})}^p\right).$$
\end{proposition}
\begin{proof}
Assume $R=[a,a+\delta]\times [b,b+\delta]$. The affine change of variables
$$(t,s)\in R\mapsto(t',s')=\eta(t,s)=(\frac{t-a}{\delta},\frac{s-b}{\delta})\in[0,1]^2$$ shows that
$$|E_Rg(x)|=\delta^2|E_{[0,1]^2}g^{a,b}(\bar{x})|,$$
$$|E_{R'}g(x)|=\delta^2|E_{R''}g^{a,b}(\bar{x})|,$$
where
$R''=\eta(R')$ is a square with side length $\frac1K$,
$$g^{a,b}(t',s')=g(\delta t'+a,\delta s'+b),$$
$$\bar{x}_1=\delta(x_1+x_3(2aA_1+2bA_2)+x_4(2aA_4+2bA_5))$$
$$\bar{x}_2=\delta(x_2+x_3(2bA_3+2aA_2)+x_4(2bA_6+2aA_5))$$
$$\bar{x}_3=\delta^2x_3,\;\;\bar{x}_4=\delta^2x_4.$$
Note that $\bar{x}$ is the image of $x$ under a shear transformation. Call $C_N$ the image of the ball $B_N$ in $\R^4$ under this transformation. Cover $C_N$ with a family $\F$ of balls $B_{\delta^2N}$ with $O(1)$ overlap. Write
$$\|E_{R}g\|_{L^p(w_{B_N})}=\delta^{2-\frac6p}\|E_{[0,1]^2}g^{a,b}\|_{L^p(w_{C_N})}$$
for an appropriate weight $w_{C_N}$. The right hand side is bounded by
$$\delta^{2-\frac6p}(\sum_{B_{\delta^2N}\in\F}\|E_{[0,1]^2}g^{a,b}\|_{L^p(w_{B_{\delta^2N}})}^p)^{1/p}.$$
Apply Proposition \ref{fp2} to each of the terms $\|E_{[0,1]^2}g^{a,b}\|_{L^p(w_{B_{\delta^2N}})}$ and then rescale back.
\end{proof}

\bigskip

We are now in position to prove Theorem \ref{ft2}. By iterating Proposition \ref{fp3} $n$ times we get
$$\|E_{[0,1]^2}g\|_{L^p(w_{B_N})}^p\le (C_pK^{p-2})^n\sum_{l(R)=\frac1{K^n}}\|E_{R}g\|_{L^p(w_{B_N})}^p+$$$$+C_pK^{4p}\sum_{l(\Delta)=N^{-1/2}}\|E_{\Delta}g\|_{L^p(w_{B_N})}^p\sum_{j=0}^{n-1}(C_pK^{p-2})^{j}D_{multi}(NK^{-2j},p,K^{-2})^p.$$
Given $\nu\le \frac1{10}$,  we apply this with $K^{-2}=\nu$ and $n$ such that $K^n= N^{\frac12}$.  We get
$$\|E_{[0,1]^2}g\|_{L^p(w_{B_N})}\le $$$$ N^{\frac1{4p}\log_{\nu^{-1}}C_p}N^{^{\frac12-\frac1p}}(1+\nu^{-2}\sum_{j=0}^{n-1}(N\nu^{j})^{\frac1p-\frac12}D_{multi}(N\nu^{j},p,\nu))(\sum_{l(\Delta)={N^{-1/2}}}\|E_{\Delta}g\|_{L^p(w_{B_N})}^p)^{1/p}
\le$$
$$ N^{\frac1{4p}\log_{\nu^{-1}}C_p}N^{^{\frac12-\frac1p}}(1+\frac{\log_{\nu^{-1}}N}{4\nu^2}\max_{1\le M\le N}M^{\frac1p-\frac12}D_{multi}(M,p,\nu))(\sum_{l(\Delta)={N^{-1/2}}}\|E_{\Delta}g\|_{L^p(w_{B_N})}^p)^{1/p}.$$
The proof of Theorem \ref{ft2} is now complete.

\bigskip

\section{The final argument}
In this section we finish the proof of Theorem \ref{ft1}, by showing that
$$D(N,6)\lesssim_\epsilon N^{\frac13+\epsilon}.$$
For $p\ge 4$ define $\kappa_p$ such that
$$\frac2p=\frac{1-\kappa_p}{2}+\frac{\kappa_p}{p},$$
in other words, $$\kappa_p=\frac{p-4}{p-2}.$$

\begin{proposition}
Let $R_1,R_2$ be $\nu$-transverse squares in $[0,1]^2$ with arbitrary side lengths.
We have that for each radius $R\ge N$, $p\ge 4$ and $g_i:R_i\to \C$
$$\|(\prod_{i=1}^2\sum_{\atop{l(\tau)=N^{-1/4}}}|E_{\tau}g_i|^2)^{\frac1{4}}\|_{L^{p}(w_{B_R})}\lesssim_{\nu,p}$$
$$
\lesssim_{\nu,p}\|(\prod_{i=1}^2\sum_{\atop{l(\Delta)=N^{-1/2}}}|E_{\Delta}g_i|^2)^{\frac1{4}}\|_{L^{p}(w_{B_R})}^{1-\kappa_p}(\prod_{i=1}^2\sum_{\atop{l(\tau)=N^{-1/4}}}\|E_{\tau}g_i\|_{L^{p}(w_{B_R})}^2)^{\frac{\kappa_p}{4}}.
$$
\end{proposition}
\begin{proof}
Let  $B$ be an arbitrary ball of radius $N^{1/2}$. We start by recalling that \eqref{we3} on $B$ gives
\begin{equation}
\label{fe20}
\|(\prod_{i=1}^2\sum_{\atop{l(\tau)=N^{-1/4}}} |E_{\tau}g_i|^2)^{1/4}\|_{L^{p}(w_{B})}\lesssim_\nu N^{-\frac2{p}}(\prod_{i=1}^2\sum_{\atop{l(\tau)=N^{-1/4}}}\|E_{\tau}g_i\|_{L^{p/2}(w_{B})}^2)^{\frac1{4}}.
\end{equation}
Write using H\"older's inequality
\begin{equation}
\label{we7}
(\sum_{\atop{l(\tau)=N^{-1/4}}}\|E_{\tau}g_i\|_{L^{p/2}(w_{B})}^2)^{\frac1{2}}\le (\sum_{\atop{l(\tau)=N^{-1/4}}}\|E_{\tau}g_i\|_{L^{2}(w_{B})}^2)^{\frac{1-\kappa_p}{2}}(\sum_{\atop{l(\tau)=N^{-1/4}}}\|E_{\tau}g_i\|_{L^{p}(w_{B})}^2)^{\frac{\kappa_p}{2}}.
\end{equation}
The next key element in our argument is the almost orthogonality specific to $L^2$, which will allow us to pass from scale $N^{-1/4}$ to scale $N^{-1/2}$. Indeed, since $(E_\Delta g_i)w_{B}$ are almost orthogonal for $l(\Delta)=N^{-1/2}$, we have
$$(\sum_{\atop{l(\tau)=N^{-1/4}}}\|E_{\tau}g_i\|_{L^{2}(w_{B})}^2)^{1/2}\lesssim (\sum_{\atop{l(\Delta)=N^{-1/2}}}\|E_{\Delta}g_i\|_{L^{2}(w_{B})}^2)^{1/2}.$$
We can now rely on the fact that $|E_{\Delta}g_i|$ is essentially constant on balls $B'$ of radius $N^{1/2}$ to argue that
$$(\sum_{\atop{l(\Delta)=N^{-1/2}}}\|E_{\Delta}g_i\|_{L^{2}(B')}^2)^{\frac1{2}}\sim |B'|^{1/2}(\sum_{\atop{l(\Delta)=N^{-1/2}}}|E_{\Delta}g_i|^2)^{\frac1{2}}|_{B'}$$
and thus
\begin{equation}
\label{we8}
(\prod_{i=1}^2\sum_{\atop{l(\Delta)=N^{-1/2}}}\|E_{\Delta}g_i\|_{L^{2}(w_{B})}^2)^{\frac1{4}}\lesssim |B|^{\frac12-\frac1p}\|(\prod_{i=1}^2\sum_{\atop{l(\Delta)=N^{-1/2}}}|E_{\Delta}g_i|^2)^{\frac1{4}}\|_{L^{p}(w_{B})}.
\end{equation}
Combining \eqref{fe20}, \eqref{we7} and \eqref{we8} we get
$$
\|(\prod_{i=1}^2\sum_{\atop{l(\tau')=N^{-1/4}}}|E_{\tau}g_i|^2)^{\frac1{4}}\|_{L^{p}(w_{B})}\lesssim_\nu \|(\prod_{i=1}^2\sum_{\atop{l(\Delta)=N^{-1/2}}}|E_{\Delta}g_i|^2)^{\frac1{4}}\|_{L^{p}(w_{B})}^{1-\kappa_p}(\prod_{i=1}^2\sum_{\atop{l(\tau)=N^{-1/4}}}\|E_{\tau}g_i\|_{L^{p}(w_{B})}^2)^{\frac{\kappa_p}{4}}.
$$
Summing this up over a finitely overlapping family of balls $B\subset B_R$ we get the desired inequality.
\end{proof}
\medskip

We will iterate the result of this proposition in the following form, a consequence of the Cauchy--Schwartz inequality
$$\|(\prod_{i=1}^2\sum_{\atop{l(\tau)=N^{-1/4}}}|E_{\tau}g_i|^2)^{\frac1{4}}\|_{L^{p}(w_{B_R})}\le$$
\begin{equation}
\label{fe21}
\le C_{p,\nu}N^{\frac{\kappa_p}{2}(\frac12-\frac1p)}\|(\prod_{i=1}^2\sum_{\atop{l(\Delta)=N^{-1/2}}}|E_{\Delta}g_i|^2)^{\frac1{4}}\|_{L^{p}(w_{B_R})}^{1-\kappa_p}(\prod_{i=1}^2\sum_{\atop{l(\tau)=N^{-1/4}}}\|E_{\tau}g_i\|_{L^{p}(w_{B_R})}^p)^{\frac{\kappa_p}{2p}}.
\end{equation}

\bigskip

We will also need the following immediate consequence of the Cauchy--Schwartz inequality. While the exponent $2^{-s}$ in $N^{2^{-s}}$ can be improved by using the bilinear Theorem \ref{ft3}, the following trivial estimate will suffice for our purposes.
\begin{lemma}
\label{wlem0081}Consider two rectangles $R_1,R_2\subset [0,1]^2$ with arbitrary side lengths. Assume $g_i$ is supported on $R_i$.
Then for $1\le p\le\infty$ and $s\ge 2$
$$\|(\prod_{i=1}^2|E_{R_i}g_i|)^{1/2}\|_{L^{p}({w_{B_N}})}\le N^{2^{-s}}\|(\prod_{i=1}^2\sum_{\atop{l(\tau_s)=N^{-2^{-s}}}}|E_{\tau_s}g_i|^2)^{\frac1{4}}\|_{L^{p}(w_{B_N})}.$$
\end{lemma}

\bigskip

Using  parabolic rescaling as in the proof of Theorem \ref{fp3}, we get that for each square $R\subset[0,1]^2$ with side length $N^{-\rho}$,  $\rho\le \frac12$
\begin{equation}
\label{fe22}
\|E_Rg\|_{L^p(w_{B_N})}\le D(N^{1-2\rho},p)(\sum_{\Delta\subset R\atop{l(\Delta)=N^{-1/2}}}\|E_\Delta g\|_{L^p(w_{B_N})}^p)^{1/p}.
\end{equation}
\bigskip

Fix a pair of $\nu$-transverse  rectangles $R_1,R_2\subset [0,1]^2$ with arbitrary side lengths and assume $g_i$ is supported on $R_i$. Start with Lemma \ref{wlem0081}, continue with iterating \eqref{fe21} $s-1$ times, and invoke \eqref{fe22} at each step to write
$$\|(\prod_{i=1}^2|E_{R_i}g_i|)^{1/2}\|_{L^{p}({B_N})}\le N^{2^{-s}}C_{p,\nu}^{s-1}(\prod_{i=1}^2\sum_{\atop{l(\Delta)=N^{-1/2}}}\|E_{\Delta}g_i\|_{L^{p}(w_{B_N})}^p)^{\frac{1}{2p}}\times$$
$$\times N^{\frac{\kappa_p}{2}(\frac12-\frac1p)(1-\kappa_p)^{s-2}}\cdot\ldots\cdot N^{\frac{\kappa_p}{2^{s-2}}(\frac12-\frac1p)(1-\kappa_p)}N^{\frac{\kappa_p}{2^{s-1}}(\frac12-\frac1p)}\|(\prod_{i=1}^2\sum_{\atop{l(\tau)=N^{-1/2}}}|E_{\Delta}g_i|^2)^{\frac1{4}}\|_{L^{p}(w_{B_N})}^{(1-\kappa_p)^{s}}\times$$
\begin{equation}
\label{fe34}
\times D(N^{1-2^{-s+1}},p)^{\kappa_p}D(N^{1-2^{-s+2}},p)^{\kappa_p(1-\kappa_p)}\cdot\ldots\cdot D(N^{1/2},p)^{\kappa_p(1-\kappa_p)^{s-2}}.
\end{equation}

Note that
$$\|(\sum_{\atop{l(\Delta)=N^{-1/2}}}|E_{\Delta}g_i|^2)^{\frac1{2}}\|_{L^{p}(w_{B_N})}\le N^{\frac12-\frac1p}(\sum_{\atop{l(\Delta)=N^{-1/2}}}\|E_{\Delta}g_i\|_{L^{p}(w_{B_N})}^p)^{1/p},$$
is an immediate  consequence of Minkowski's and H\"older's inequalities. Using this, \eqref{fe34} can be rewritten as follows 
$$D_{multi}(N,p,\nu)\le C_{p,\nu}^{s-1} N^{2^{-s}}N^{\kappa_p 2^{-s}(1-\frac2p)\frac{1-(2(1-\kappa_p))^{s-1}}{2\kappa_p-1}}\times$$
\begin{equation}
\label{fe23}
\times D(N^{1-2^{-s+1}},p)^{\kappa_p}D(N^{1-2^{-s+2}},p)^{\kappa_p(1-\kappa_p)}\cdot\ldots\cdot D(N^{1/2},p)^{\kappa_p(1-\kappa_p)^{s-2}}    N^{O_p((1-\kappa_p)^s)}.
\end{equation}
\bigskip

Let $\gamma_p$ be the unique positive number such that
$$\lim_{N\to\infty}\frac{D(N,p)}{N^{\gamma_p+\epsilon}}=0,\;\text{for each }\epsilon>0$$
and
\begin{equation}
\label{fe27}
\limsup_{N\to\infty}\frac{D(N,p)}{N^{\gamma_p-\epsilon}}=\infty,\;\text{for each }\epsilon>0.
\end{equation}
The existence of such $\gamma_p$ is guaranteed by \eqref{fe30} and \eqref{fe3008}.
Recall that our goal is to prove that $\gamma_6=\frac13.$
By using the fact that $D(N,p)\lesssim_\epsilon N^{\gamma_p+\epsilon}$ in \eqref{fe23}, it follows that for each $\nu$ , $\epsilon>0$ and $s\ge 2$
\begin{equation}
\label{fe32}
\limsup_{N\to\infty}\frac{D_{multi}(N,p,\nu)}{N^{\gamma_{p,\epsilon,s}}}<\infty
\end{equation}
 where
$$\gamma_{p,\epsilon,s}=2^{-s}+\kappa_p(\gamma_p+\epsilon)(\frac{1-(1-\kappa_p)^{s}}{\kappa_p}-2^{-s+1}\frac{1-(2(1-\kappa_p))^{s}}{2\kappa_p-1})+$$$$+\kappa_p 2^{-s}(1-\frac2p)\frac{1-(2(1-\kappa_p))^{s-1}}{2\kappa_p-1}+O_p((1-\kappa_p)^s).$$

We will show now that if $p>6$ then
$$\gamma_p\le \frac{p-6}{2p-8}+\frac12-\frac1p.$$
If we manage to do this, it will suffice to let $p\to 6$ to get $\gamma_6\le \frac13$, hence $\gamma_6=\frac13$, as desired. We first note that
\begin{equation}
\label{fe24}
2(1-\kappa_p)=\frac4{p-2}<1
\end{equation}
Assume for contradiction that
\begin{equation}
\label{fe26}
\gamma_p>\frac{p-6}{2p-8}+\frac12-\frac1p.
\end{equation}
A simple computation using \eqref{fe24} and \eqref{fe26} shows that for $s$ large enough, and $\epsilon$ small enough we have
$\gamma_{p,\epsilon,s}<\gamma_p.$

Choose now $\nu$ so small that
\begin{equation}
\label{fe33} 
\gamma_{p,\epsilon,s}':=\gamma_{p,\epsilon,s}+\epsilon(\nu,p)<\gamma_p,
\end{equation}
and 
\begin{equation}
\label{fe35}
\frac12-\frac1p+\epsilon(\nu,p)<1-\frac4p
\end{equation}
 where $\epsilon(\nu,p)$ is from Theorem \ref{ft2}. The values of $\epsilon,s,\nu$ are fixed for the rest of the argument.

We will argue that
$$\limsup_{N\to\infty}\frac{D(N,p)}{N^{\max\{\gamma_{p,\epsilon,s}',\frac12-\frac1p+\epsilon(\nu,p)\}}}<\infty.$$
This will contradict either \eqref{fe27} (in light of  \eqref{fe33}) or  \eqref{fe30} (in light of  \eqref{fe35}). We distinguish two cases.

If $\gamma_{p,\epsilon,s}\le \frac12-\frac1p$, then \eqref{fe32} and \eqref{fe31} lead to
$$\limsup_{N\to\infty}\frac{D(N,p)}{N^{\frac12-\frac1p+\epsilon(\nu,p)}}<\infty.$$
If $\gamma_{p,\epsilon,s}\ge \frac12-\frac1p$, then \eqref{fe32} and \eqref{fe31} lead to
$$\limsup_{N\to\infty}\frac{D(N,p)}{N^{\gamma_{p,\epsilon,s}'}}<\infty.$$
In conclusion, inequality \eqref{fe26} can not hold, and the proof of Theorem \ref{ft1} is complete.

\label{fslast}

\end{document}